\documentclass[11pt]{article}

\usepackage[T1]{fontenc}
\usepackage{lmodern}
\usepackage{microtype}
\usepackage[a4paper,margin=30mm]{geometry}
\usepackage{amsmath,amssymb,amsthm,mathtools}
\usepackage{xcolor}
\usepackage{enumitem}
\usepackage[authoryear,round]{natbib}
\usepackage[hidelinks]{hyperref}

\definecolor{heading}{HTML}{17365D}
\definecolor{accent}{HTML}{2F6B8A}
\definecolor{soft}{HTML}{F2F6F8}

\hypersetup{
  pdftitle={Positivity of Smooth Currents on Singular Spaces},
  pdfsubject={Smooth positive currents, tangent cones, and a counterexample on a singular threefold}
}

\newtheoremstyle{clean}
  {8pt}{8pt}{\itshape}{}{{\color{heading}\bfseries}}{.}{0.5em}{}
\theoremstyle{clean}
\newtheorem{theorem}{Theorem}
\newtheorem{lemma}[theorem]{Lemma}
\newtheorem{proposition}[theorem]{Proposition}
\newtheorem{definition}[theorem]{Definition}

\newtheorem{corollary}[theorem]{Corollary}

\newcommand{\C}{\mathbb C}
\newcommand{\Pj}{\mathbf P}
\newcommand{\ddc}{dd^{c}}
\newcommand{\reg}{\mathrm{reg}}
\newcommand{\HS}{\mathrm{HS}}
\newcommand{\tr}{\operatorname{tr}}
\newcommand{\rank}{\operatorname{rank}}

\newcommand{\PSH}{\operatorname{PSH}}
\providecommand{\subject}[1]{\textbf{\textit{Mathematics Subject Classification 2020:}} #1}

\title{Positivity of Smooth Currents on Singular Spaces}
\author{Duc-Thai Do \and Duc-Viet Vu}
\newcommand{\Addresses}{{
		\bigskip
		\footnotesize
		\noindent
			\textsc{Duc-Thai Do, Department of Mathematics, Hanoi National University of Education, 136 XuanThuy str., Hanoi, Vietnam.}
		\noindent
		\par\nopagebreak
		\noindent
		\textit{E-mail address}: \texttt{doducthai@hnue.edu.vn}	
\newline
\vskip0.3cm
	\noindent
			\textsc{Duc-Viet Vu, University of Cologne, Division of Mathematics, Department of Mathematics and Computer Science, Weyertal 86-90, 50931, K\"oln,  Germany.}
		\noindent
		\par\nopagebreak
		\noindent
		\textit{E-mail address}: \texttt{dvu@uni-koeln.de}

}}
\date{}

\begin{document}

\maketitle
\thispagestyle{plain}

\begin{abstract} 
We study two notions of positivity for smooth currents on singular spaces. We show that they are not equivalent by establishing a necessary condition on the fourth Whitney cones. We also construct an explicit compact normal projective variety $X$ and a real smooth $(1,1)$-form $T$ on $X$ such that $T$ has local smooth $\ddc$-potentials and is a K\"ahler current, but $T$ is not a positive form (hence not Hermitian). 
\end{abstract}


\noindent
\subject{32U15}, {32Q15}, {53C55}.

\section{Introduction}\label{sec:introduction}

On a singular complex space, there are two natural positivity tests for a smooth $(1,1)$-form. The first one is to consider the form as a current and test whether the form is a positive current. The second one is the positivity inherited from local embeddings (see Section \ref{sec:prelim} for details). In the smooth setting, the two notions are the same. The purpose of this note is to study these notions in the singular setting and give an explicit example showing that they are different in general.

We use the following terminology.  A \emph{Hermitian metric} (resp. \emph{positive form}) on a complex space $X$ is a smooth real $(1,1)$-form which, under local embeddings of $X$ into a complex manifold, is the restriction of a strictly positive smooth $(1,1)$-form (resp. a positive form).    A \emph{K\"ahler
current on $X$} is a closed positive $(1,1)$-current $S$ satisfying
\[
  S\geq \delta\omega
\]
for some Hermitian metric $\omega$ on $X$ and some constant $\delta>0$.

Let $X$ be a complex space and $x\in X$. A smooth form on $X$ can be viewed as a current on $X.$ When we want to emphasize this point of view, we will refer to smooth forms as smooth currents.  Let $T_x X$ be the tangent space of $X$ at $x$ (see Section \ref{subsec:tangent-cones} for its definition). Let $\theta$ be a real smooth $(1,1)$-form  on $X$. For any subset $E\subset T_xX$, we say that $\theta$ is \emph{E-positive at $x$} if for every local smooth extension $\theta'$ of $\theta$ under a local embedding of $(X,x)$, we have $\theta'(x)(v, \overline v) \ge 0$ for every $v \in E$. Let $C_4(X,x)$ be the fourth Whitney cone of $X$ at $x$ (see Definition \ref{def:fourth-whitney-cone} below). We recall that $C_4(X,x)$ is an algebraic subset of $T_xX$ and if $x$ is a regular point, then $C_4(X,x)=T_x X$.  We will prove later that for a pure-dimensional complex space $X$, the form $\theta$ is positive as a current on $X$ if and only if it is $C_4(X,x)$-
positive at $x$ for every $x\in X$ (see Lemma \ref{le-tangentialpositive} below).  Our first main result is the following.

\begin{theorem}\label{thm-tangentialpositive} Let $X$ be a complex space. Let $\theta$ be a smooth real $(1,1)$-form on $X$.  Let $x_0\in X$. Assume that $C_4(X,x_0)$ has an irreducible component which is not contained in any complex hyperplane of $T_{x_0}X$, and there exists a local smooth extension $\theta'$ of $\theta$  under a local embedding of $(X,x_0)$ such that  $\theta'(x_0)|_{T_{x_0}X}$ is not semipositive. Then, $\theta$ is not positive as a smooth form.
\end{theorem}

The above result is a first sign showing the non-equivalence between the positivity as currents and as forms for smooth currents because the cone $C_4(X,x)$ is in general a proper subset of $T_xX$ for a singular point $x \in X$. The second main result is an explicit sufficient condition for the above-mentioned phenomenon. 

\begin{theorem}\label{thm-tangentialpositive2} Let $X$ be a compact pure-dimensional Kähler space. Let $x_0\in X$. Assume that the following conditions hold:

(i) $\dim C_4(X, x_0) < \dim T_{x_0}X$,

(ii) $C_4(X,x_0)$ has an irreducible component which is not contained in any complex hyperplane of $T_{x_0}X$. 

Then there exists a real closed smooth $(1,1)$-form $\theta$ having smooth local potentials on $X$ such that $\theta$ is a Kähler current on $X$ but is not positive as a smooth form. In particular, such a $\theta$ is not a Kähler form.  %there exists a local smooth extension $\theta'$ of $\theta$  under a local embedding of $(X,x_0)$ such that  $\theta'(x_0)|_{T_{x_0}X}$ is not definite semi-positive, then $\theta$ is not positive as a smooth form.
\end{theorem}

The  idea  of the proof of Theorem \ref{thm-tangentialpositive2} is as follows. Using the  strict inclusion of $C_4(X,x)$ in $T_xX$, we will construct a local constant Hermitian form $H$ in a local embedding of  $(X,x)$ such that $H(x)$ is an indefinite Hermitian form on $T_xX$ but it is  positive definite along $C_4(X,x)$ in the Euclidean space. Because of the hypothesis on $C_4(X,x)$, we will show that  every other Hermitian form, which is equal to $H$ on that cone, cannot be positive definite on the whole space. This treatment explains the existence of an infinitesimal example. To have a global example, we use a standard cut-off process. %consider a concrete example but the previously mentioned discussion explain the underlying picture.    
In Section \ref{sec:compact-model}, we will present a concrete example of a projective normal space $X$ satisfying the assumption of Theorem \ref{thm-tangentialpositive2} (see Corollary \ref{cor:compact-projective-example} there). 

In Section~\ref{sec:prelim}, we recall notions needed for the construction:  forms, currents, functions on singular spaces
and quotient spaces. Sections \ref{sec-cone-dauam} and \ref{subsec:positivity-regular-locus} are devoted to a proof of Theorems \ref{thm-tangentialpositive} and \ref{thm-tangentialpositive2}. 
\\

\noindent
\textbf{Acknowledgment.} The research of Duc-Thai Do is supported by the ministry-level project B2025-SPH-02. The research of Duc-Viet Vu is partially funded by the DFG-Project VU 126/1-2.  
An initial raw example presented in  Corollary \ref{cor:compact-projective-example} was found by ChatGPT and subsequently refined/expanded by the authors.

\section{Preliminaries}\label{sec:prelim}

All complex spaces in this note are reduced.  General references are \citet[Chapters~1--2]{Chirka1989} and \cite{Demailly1985}.

\subsection{Forms, currents, and plurisubharmonic functions on singular spaces}
\label{subsec:currents-on-spaces}

We recall the notions of smooth forms, currents on singular spaces.  The basic reference is
\citet[\S1]{Demailly1985}; see also
\citet[Chapters~II and III]{Demailly2012}.  The same conventions are used in
recent work on singular K\"ahler spaces, for example
\citet[Definition~2.28 and \S2.4.1]{DasHaconPaun2024} and
\citet[\S4.1]{HaconPaun2024}.

Let $X$ be a complex space of dimension $n$.  Locally, one may choose a holomorphic
embedding
\[
  \jmath:U\hookrightarrow\Omega\subset\C^N
\]
whose image is a closed analytic subset of the open set $\Omega$.  A smooth
$(a,b)$-form on $U$ is, by definition, the restriction $\jmath^*\alpha$ of a
smooth $(a,b)$-form $\alpha$ on $\Omega$.  Thus the space of smooth forms on
$U$ is the image of the restriction map
\[
  \mathcal A^{a,b}(\Omega)\longrightarrow \mathcal A^{a,b}(U),
  \qquad \alpha\longmapsto\jmath^*\alpha,
\]
equipped with the quotient topology.  This definition does not depend on the
chosen embedding.  The operators
$d,\partial,\bar\partial$ descend from the ambient space.
%Assume for the moment that $X$ has pure dimension $n$, and 
Let
$\mathcal D^{a,b}(X)$ denote compactly supported smooth $(a,b)$-forms with
their usual topology.  A current of \emph{bidimension}
$(q,r)$ is a continuous linear functional on $\mathcal D^{q,r}(X)$.  The same
current is said to have \emph{bidegree} $(n-q,n-r)$. 
The differential of a current is defined by duality, with the usual sign.
A current is \emph{closed} if $dT=0$; for a current of pure type this is
equivalent to $\partial T=\bar\partial T=0$.

For a local embedding $\jmath:U\hookrightarrow\Omega$, the direct image
\[
  \langle \jmath_*T,\eta\rangle
  :=\langle T,\jmath^*\eta\rangle
\]
identifies currents on $U$ with a distinguished subspace of ambient currents
supported on $\jmath(U)$.  It is injective and commutes with
$d,\partial,\bar\partial$. 

A smooth $(k,k)$-form is \emph{(strongly) positive} if it is locally the restriction of (strongly) positive forms on ambient spaces under local embeddings.   
A current $T$ of bidegree
$(1,1)$ on $X$ is \emph{positive}, written $T\geq0$,
if
\[
  \langle T,\Phi\rangle\geq 0
\]
for every compactly supported positive $(n-1,n-1)$-form $\Phi$.
Equivalently, in every local embedding, $\jmath_*T$ is a positive ambient
current.  Positive currents have order zero: their coefficient distributions
are complex measures.

A function $u:X\to[-\infty,+\infty)$, not identically $-\infty$ on any open
set, is \emph{plurisubharmonic}, or psh, if it is locally the restriction of
an ambient psh function: for every local embedding
$\jmath:U\hookrightarrow\Omega$, after shrinking $U$ if necessary, there is
$\widetilde u\in\PSH(\Omega)$ such that
$u=\widetilde u|_U$.  The theorem of \citet[Theorem~5.3.1]{FornaessNarasimhan1980}
says equivalently that $u$ is upper semicontinuous and $u\circ f$ is
subharmonic or identically $-\infty$ for every holomorphic disc
$f:\mathbb D\to X$; see also \citet[Definition~1.5 and
Theorem~1.6]{Demailly1985}. A closed positive $(1,1)$-current $T$ on $X$ is said to have \emph{local potentials} if locally $T= \ddc u$ for some local psh function $u$ on $X$ (recall that $d^c:=\frac{i}{2 \pi} (\bar \partial -\partial)$).

A \emph{Kähler form} on \(X\) is a smooth real \((1,1)\)-form \(\omega\) which locally admits smooth strictly plurisubharmonic potentials, i.e. locally \(\omega=dd^c\rho\) with \(\rho\) smooth and strictly plurisubharmonic. A complex space admitting a Kähler form is called a Kähler space.

\subsection{Tangent cones}
\label{subsec:tangent-cones}

We start with some definitions about cones. 
\begin{definition}
A nonempty subset \(C\subset \mathbb C^n\) is called a
\emph{complex cone} if
\[
  \lambda z\in C
  \qquad
  \text{for every } z\in C
  \text{ and every } \lambda\in\mathbb C.
\]
Equivalently, \(0\in C\) and
\[
  \lambda C=C
  \qquad
  \text{for every } \lambda\in\mathbb C^*.
\]
The cone \(C\) is called \emph{proper} if
\[
  C\neq \mathbb C^n.
\]
\end{definition}
We discuss now several notions of tangent cones for singular spaces.
Let $X$
be a reduced complex space, let $x\in X$, and choose a local embedding
\[
  \jmath:(X,x)\hookrightarrow(\C^N,0).
\]
All cones below are written in these coordinates.  Their transformation
under biholomorphic changes of embedded germs is governed by the differential
at the base point; see \citet[Chapter~2, \S 8.1--8.2 and
9.1--9.2]{Chirka1989}.
We define the \emph{tangent space} of $X$ at $x$ by
\[
  T_xX:=\left(\frac{\mathfrak m_{X,x}}
                    {\mathfrak m_{X,x}^{2}}\right)^{\!*}.
\]
Denote by $\mathcal I_{X,0}\subset\mathcal O_{\C^N,0}$ the ideal of
holomorphic germs vanishing on the embedded germ $\jmath(X,x)$.

\begin{lemma}
\label{lem:embedded-zariski-tangent}
The differential of $\jmath$ identifies $T_xX$ canonically with
\begin{equation}
\label{eq:embedded-zariski-tangent}
  \bigl\{v\in\C^N:df_0(v)=0
  \text{ for every }f\in\mathcal I_{X,0}\bigr\}.
\end{equation}
\end{lemma}

\begin{proof}
Let $z_1,\ldots,z_N$ be the standard coordinates on $\C^N$, and write
$\zeta_j:=z_j|_X$ for their restrictions to the embedded germ.  Let
$K$ denote the vector space in \eqref{eq:embedded-zariski-tangent}.  For
$v\in K$ and $g\in\mathfrak m_{X,x}$, choose a holomorphic ambient extension
$\widetilde g\in\mathfrak m_{\C^N,0}$ and set
\begin{equation}
\label{eq:lambda-v-definition}
  \lambda_v([g]):=d\widetilde g_0(v),
  \qquad [g]\in\mathfrak m_{X,x}/\mathfrak m_{X,x}^2.
\end{equation}
This is well defined.  Indeed, two ambient extensions differ by a germ in
$\mathcal I_{X,0}$, whose differential annihilates $v$ by the definition of
$K$.  Moreover, if $g\in\mathfrak m_{X,x}^2$, write locally
$g=\sum_\nu g_\nu h_\nu$ with $g_\nu(x)=h_\nu(x)=0$, and choose ambient
extensions $\widetilde g_\nu,\widetilde h_\nu$ vanishing at the origin.  The
ambient germ $\sum_\nu\widetilde g_\nu\widetilde h_\nu$ extends $g$ and
has zero differential at the origin.  Any other ambient extension differs
from it by an element of $\mathcal I_{X,0}$, whose differential annihilates
$v$.  Thus \eqref{eq:lambda-v-definition} depends only on the class of $g$
modulo $\mathfrak m_{X,x}^2$ and defines an element of $T_xX$.
The resulting map $K\to T_xX$, $v\mapsto\lambda_v$, is injective: if
$\lambda_v=0$, then
\[
  v_j=(dz_j)_0(v)=\lambda_v([\zeta_j])=0,
  \qquad 1\leq j\leq N,
\]
so $v=0$.
Conversely, let $\lambda\in T_xX$ and define
\[
  v:=\bigl(\lambda([\zeta_1]),\ldots,
           \lambda([\zeta_N])\bigr)\in\C^N.
\]
If $f\in\mathcal I_{X,0}$, then the linear part of $f$ gives
\[
  df_0(v)
  =\lambda\!\left(\left[f|_X\right]\right)=0,
\]
because $f|_X=0$.  Hence $v\in K$.  For an arbitrary
$g\in\mathfrak m_{X,x}$, the class of an ambient extension $\widetilde g$
modulo the square of the maximal ideal is its linear part, and therefore
\[
  \lambda_v([g])=d\widetilde g_0(v)=\lambda([g]).
\]
Thus $K\to T_xX$ is surjective as well.  
\end{proof}

\begin{definition}[Ordinary tangent cone and analytic tangent space]
\label{def:ordinary-tangent-cone}
The \emph{ordinary tangent cone} of $X$ at $x$ is
\begin{equation}
\label{eq:sequential-tangent-cone}
  C_0(X,x):=\left\{v\in\C^N:
  \begin{array}{l}
  \text{there are }x_\nu\in X,\ x_\nu\to x,\text{ and }t_\nu>0\\[-2pt]
  \text{such that }t_\nu\jmath(x_\nu)\to v
  \end{array}\right\}.
\end{equation}
Equivalently, let $\mathcal I_{X,x}\subset\mathcal O_{\C^N,0}$ be the ideal
of germs vanishing on the embedded germ.  If
\[
  f=f_m+f_{m+1}+\cdots,
  \qquad f_m\neq0,
\]
is the Taylor expansion into homogeneous polynomials, write
$\operatorname{in}(f):=f_m$.  Then
\begin{equation}
\label{eq:initial-form-tangent-cone}
  C_0(X,x)=\bigl\{v\in\C^N:
  \operatorname{in}(f)(v)=0\text{ for every }
  f\in\mathcal I_{X,x}\bigr\}.
\end{equation}
Thus $C_0(X,x)\subset T_xX$ (here we identify $T_xX$ with a subspace of $\C^N$), but the ordinary tangent cone need not be linear.
\end{definition}

For an analytic germ of pure dimension $p$, the ordinary tangent cone is a complex algebraic cone of pure dimension $p$.  Moreover, $x$ is regular if
and only if $T_xX$ has dimension $p$; hence at every singular point
\begin{equation}
\label{eq:tangent-space-dimension-jump}
  \dim_{\C}T_xX>p.
\end{equation}
These facts, together with the equivalence of
\eqref{eq:sequential-tangent-cone} and
\eqref{eq:initial-form-tangent-cone}, are proved analytically in
\citet[Chapter~2, \S\S8.1 and 8.4]{Chirka1989}.

The cone relevant to limits of tangent \emph{spaces} along the regular locus
is larger than the ordinary tangent cone.

\begin{definition}[Fourth Whitney tangent cone]
\label{def:fourth-whitney-cone}
The \emph{fourth Whitney tangent cone} at
$x$ of $X$ is
\begin{equation}
\label{eq:C4-definition}
  C_4(X,x):=\left\{v\in\C^N:
  \begin{array}{l}
  \text{there are }x_\nu\in X_{\reg},\ x_\nu\to x,\text{ and}\\[-2pt]
  v_\nu\in T_{x_\nu}X_{\reg}\text{ with }v_\nu\to v
  \end{array}\right\}.
\end{equation}
It is a closed complex cone;  see
\citet[Chapter~2, \S\S9.1--9.2]{Chirka1989}.
\end{definition}
 Let \begin{equation}
\label{eq:local-embedding-dimension}
  d(x):=\dim_{\C}\frac{\mathfrak m_{X,x}}
                         {\mathfrak m_{X,x}^{2}}.
\end{equation}
Observe that $C_4(X,x)$ is an algebraic subset of $T_xX$ (see \cite{Chirka1989}). The above constructions of $C_0(X,x)$ and $C_4(X,x)$ depend on a local embedding. However, these cones are intrinsic subsets of $T_xX$ by the existence of the minimal local embedding in the following result. 

\begin{theorem}[{\citet[Page~333]{Grauert1962} or \citet{Narasimhan1962}}]
\label{thm:grauert-local-extension}
Let $x$ be a point of a reduced complex space $X$. Then the following assertions hold:

(i) There exist a neighborhood $U$ of $x$, an open set
  $G\subset\C^{d(x)}$, an analytic subset $A\subset G$, and a
  biholomorphism
  \[
    \tau:U\longrightarrow A.
  \]

 (ii) Let $(U,\tau,A)$ be any such chart, write $z:=\tau(x)$, and let
  \[
    \psi:U\longrightarrow\C^n
  \]
  be a local embedding at $x$.  Then there are an open neighborhood
  $V\subset G$ of $z$ and a holomorphic embedding
  \[
    \widetilde\psi:V\longrightarrow\C^n
  \]
  such that, with $W:=\tau^{-1}(A\cap V)$,
  \begin{equation}
  \label{eq:grauert-extension-identity}
    \widetilde\psi\circ\tau=\psi
    \qquad\text{on }W.
  \end{equation}
\end{theorem}

\subsection{Finite analytic quotients} \label{subsec:finite-quotient-normal}

We recall in this subsection several standard facts about quotient spaces. 
Let $H$ be a finite group of biholomorphisms of a complex space $M$.  Its
analytic quotient $q:M\to M/H$ is the orbit space equipped with the sheaf
\begin{equation}\label{eq:analytic-quotient-sheaf}
  \mathcal O_{M/H}(V)
  :=\mathcal O_M\bigl(q^{-1}(V)\bigr)^H.
\end{equation}
By \citet[Theorem 4]{Cartan1957}, $M/H$ is a complex space and is also normal if $M$ is normal. We see that  holomorphic functions downstairs are exactly the
$H$-invariant holomorphic functions upstairs.  The map $q$ is finite and
proper.  For the analytic quotient theorem in this setting, see
\citet{Cartan1957}; finite holomorphic maps are treated systematically in
\citet[pp.~61--74]{GrauertRemmert1984}.

\begin{proposition}[{\citet[Corollary~3.46(2), p.~103]{Viehweg-book}}]
\label{prop:finite-quotient-projective} 
Let $X$ be a projective reduced complex space, and let $G$ be a finite
group acting on $X$ by biholomorphisms. Then the analytic quotient
$ Y:=X/G$ is projective.
\end{proposition}

\section{Cones and Hermitian metrics} \label{sec-cone-dauam}

\begin{lemma}[Hermitian separation of a proper complex cone]
Let \(C\subset\mathbb C^n\) be a closed proper complex cone.
Then there exist a Hermitian sesquilinear form \(H\) on
\(\mathbb C^n\) and a constant \(c>0\) such that
\[
  H(z,z)\geq c\lVert z\rVert^2
  \qquad
  \text{for every } z\in C.
\]
In particular,
\[
  H(z,z)>0
  \qquad
  \text{for every } z\in C\setminus\{0\}.
\]
Moreover, \(H\) can be chosen to have signature \((n-1,1)\).
\end{lemma}

\begin{proof}
If \(C=\{0\}\), the positivity assertion is vacuous, and any
Hermitian form of signature \((n-1,1)\) has the desired property.
We may therefore assume that \(C\neq\{0\}\).

Since \(C\) is proper, we may choose
\[
  a\in\mathbb C^n\setminus C,
  \qquad
  \lVert a\rVert=1.
\]
Because \(C\) is a complex cone, one has
\[
  \mathbb Ca\cap C=\{0\}.
\]
Indeed, if \(\lambda a\in C\) for some \(\lambda\neq0\), then
\[
  a=\lambda^{-1}(\lambda a)\in C,
\]
which contradicts the choice of \(a\).

Let
\[
  K:=C\cap S^{2n-1},
  \qquad
  S^{2n-1}:=
  \{z\in\mathbb C^n:\lVert z\rVert=1\}.
\]
Since \(C\) is closed, \(K\) is compact. With respect to the
standard Hermitian inner product
\[
  \langle z,w\rangle
  :=
  \sum_{j=1}^n z_j\overline{w_j},
\]
define
\[
  \rho:=
  \max_{u\in K}|\langle u,a\rangle|^2.
\]
We claim that \(\rho<1\). Otherwise, there would exist \(u\in K\)
such that
\[
  |\langle u,a\rangle|=1.
\]
Since \(\lVert u\rVert=\lVert a\rVert=1\), equality in the
Cauchy--Schwarz inequality implies that
\[
  u=e^{\sqrt{-1}\theta}a
\]
for some \(\theta\in\mathbb R\). Since \(u\in C\) and \(C\) is a
complex cone, it would follow that
\[
  a=e^{-\sqrt{-1}\theta}u\in C,
\]
a contradiction.

Set
\[
  A:=\frac{2}{1+\rho}.
\]
Since \(0\leq\rho<1\), we have
\[
  A>1
  \qquad\text{and}\qquad
  A\rho<1.
\]
Define
\[
  H(z,w)
  :=
  \langle z,w\rangle
  -
  A\langle z,a\rangle
  \overline{\langle w,a\rangle}.
\]
This is a Hermitian sesquilinear form.

Let \(z\in C\setminus\{0\}\), and set
\[
  u:=\frac{z}{\lVert z\rVert}\in K.
\]
Then
\begin{align*}
  H(z,z)
  &=
  \lVert z\rVert^2
  -
  A|\langle z,a\rangle|^2 \\
  &=
  \lVert z\rVert^2
  \left(
    1-A|\langle u,a\rangle|^2
  \right) \\
  &\geq
  (1-A\rho)\lVert z\rVert^2.
\end{align*}
Furthermore,
\[
  1-A\rho
  =
  1-\frac{2\rho}{1+\rho}
  =
  \frac{1-\rho}{1+\rho}>0.
\]
Thus the conclusion holds with
\[
  c:=\frac{1-\rho}{1+\rho}.
\]

Finally, write
\[
  z=z^\perp+\lambda a,
  \qquad
  z^\perp\perp a.
\]
Then
\[
  H(z,z)
  =
  \lVert z^\perp\rVert^2+(1-A)|\lambda|^2.
\]
Hence \(H\) is positive definite on \(a^\perp\), while
\[
  H(a,a)=1-A<0.
\]
Therefore \(H\) has signature \((n-1,1)\).
\end{proof}

\begin{lemma}\label{lem:hermitian-uniqueness-cone}
Let $C_0$ be an embedded connected complex submanifold  of $\C^n$ such that it is   not contained in a complex
hyperplane of $\C^n$. Let \(H\) and \(H'\) be Hermitian forms on
\(\mathbb C^n\) such that
\[
  H'(z,z)=H(z,z)
  \qquad
  \text{for every }z\in C_0.
\]
Then \(H'=H\). %Consequently, if \(H\) has signature \((n-1,1)\), then \(H'\) cannot be positive definite.
\end{lemma}

\begin{proof}
Set
\[
  B:=H'-H.
\]
Then \(B\) is a Hermitian form satisfying
\[
  B(z,z)=0
  \qquad
  \text{for every }z\in C_0.
\]
Write
\[
  B(z,w)
  =
  \sum_{j,k=1}^{n}b_{jk}z_j\overline{w_k},
  \qquad
  b_{jk}=\overline{b_{kj}},
\]
and consider the complex bilinear polynomial
\[
  \widetilde B(z,\zeta)
  :=
  \sum_{j,k=1}^{n}b_{jk}z_j\zeta_k,
  \qquad
  (z,\zeta)\in\mathbb C^n\times\mathbb C^n.
\]
Thus
\[
  B(z,z)=\widetilde B(z,\overline z).
\]
Let   $ \varphi:U\longrightarrow C_0$
be a local holomorphic parametrization of $C_0$, where
\(U\subset\mathbb C^{k}\) is open and $k=\dim C_0$. Set
\[
  \varphi^\#(s):=\overline{\varphi(\overline s)}.
\]
After shrinking \(U\), this defines a holomorphic parametrization
of the conjugate manifold \(\overline{C_0}\).
The function
\[
  F(t,s)
  :=
  \widetilde B\bigl(\varphi(t),\varphi^\#(s)\bigr)
\]
is holomorphic in \((t,s)\). For \(s=\overline t\), we have
\[
  F(t,\overline t)
  =
  \widetilde B\bigl(\varphi(t),\overline{\varphi(t)}\bigr)
  =
  B(\varphi(t),\varphi(t))
  =
  0.
\]
The set
\[
  \{(t,\overline t):t\in U\}
\]
is maximally totally real in
\(\mathbb C^{k}\times\mathbb C^{k}\). It follows that
\[
  F\equiv0.
\]
By analytic continuation,
\[
  \widetilde B(z,\zeta)=0
  \qquad
  \text{for every }
  (z,\zeta)\in C_0\times\overline{C_0}.
\]
Since \(C_0\) is not contained in any complex hyperplane, we get
\[
  \operatorname{Span}_{\mathbb C}(C_0)=\mathbb C^n.
\]
Similarly,
\[
  \operatorname{Span}_{\mathbb C}(\overline{C_0})=\mathbb C^n.
\]
Since \(\widetilde B\) is complex bilinear and vanishes on
\(C_0\times\overline{C_0}\), it follows that
\[
  \widetilde B\equiv0
  \qquad
  \text{on }\mathbb C^n\times\mathbb C^n.
\]
Therefore \(B=0\), and hence  $H'=H.$
\end{proof}

\section{Positivity of smooth forms}
\label{subsec:positivity-regular-locus}

In this section, we prove  Theorems \ref{thm-tangentialpositive} and \ref{thm-tangentialpositive2}. The following lemma follows directly from the definition of positive currents. 
\begin{lemma}
\label{lem:positivity-regular-locus}
Let $X$ be a complex space of pure dimension $n$, let $\theta$ be a
smooth real $(1,1)$-form on $X$, and let $\omega$ be a Hermitian metric on
$X$.
\begin{enumerate}[label=\textup{(\roman*)}]
\item The smooth form $\theta$ is positive as a current on $X$ if and only
  if $\theta|_{X_{\reg}}$ is pointwise semipositive.
\item For every $\delta>0$,
  \[
    \theta\geq\delta\omega\quad\text{as currents on }X
    \quad\Longleftrightarrow\quad
    \theta|_{X_{\reg}}\geq
      \delta\omega|_{X_{\reg}}\quad\text{pointwise}.
  \]
  Consequently, if $d\theta=0$, then $\theta$ is a K\"ahler current precisely when its
  restriction to $X_{\reg}$ is \emph{uniformly K\"ahler relative to $X$},
  meaning that the displayed pointwise inequality holds for some Hermitian
  metric $\omega$ on $X$ and one constant $\delta>0$.
\end{enumerate}
In particular, a smooth K\"ahler current restricts to a K\"ahler form on
$X_{\reg}$.
\end{lemma}

One should note that in Lemma \ref{lem:positivity-regular-locus}, $X$ must be of pure dimension because otherwise the current associated with a smooth form will always be zero on irreducible components of $X$ of dimension strictly less than $\dim X$.  

\begin{proof}
A smooth $(1,1)$-form on a reduced pure-dimensional complex space defines a
current by integration over the regular locus:
\[
  \langle\theta,\Phi\rangle
  =\int_{X_{\reg}}\theta\wedge\Phi,
\]
where $\Phi$ is a compactly supported test form of bidegree $(n-1,n-1)$.
If $\theta|_{X_{\reg}}$ is semipositive, then the integrand is a nonnegative
measure for every strongly positive $\Phi$, and hence $\theta$ is a positive
current.

Conversely, assume that the current defined by $\theta$ is positive.  Let
$V\Subset X_{\reg}$ be a coordinate ball.  Every compactly supported
strongly positive test form on $V$ is also a test form on $X$.  The
restriction of the current to $V$ is therefore positive.  Since its
coefficients are smooth on the manifold $V$, the standard localization test
for smooth forms shows that $\theta|_V$ is pointwise semipositive: otherwise
one could choose a nonnegative bump function supported near a negative
direction and obtain a negative pairing.  Such coordinate balls cover
$X_{\reg}$, proving \textup{(i)}.

Apply \textup{(i)} to the smooth form $\theta-\delta\omega$.  This gives the
equivalence in \textup{(ii)}.  If $d\theta=0$, the definition of a
K\"ahler current gives the final assertion.  This argument is the
smooth-form specialization of the positivity conventions in
\citet[Chapter~III, \S1]{Demailly2012}.  
\end{proof}

Now we prove another equivalence criterion for being positive as currents. 

\begin{lemma}
  \label{le-tangentialpositive}
  Let $X$ be a complex space of pure dimension $n$, let $\theta$ be a
smooth real $(1,1)$-form on $X$. Then,  $\theta$ is positive as a current on $X$ if and only if it is $C_4(X,x)$-
positive at $x$ for every $x\in X$.
\end{lemma}

\begin{proof}
     Fix $x\in X$ and a local embedding
\[
  \jmath:U\hookrightarrow\Omega\subset\C^N
\]
with $\jmath(x)=0$, and let $\widetilde\theta$ be an arbitrary smooth
ambient extension of $\theta$.

We first record the elementary observation that explains the appearance
of the fourth Whitney cone. Suppose that $\widetilde\theta_1$ and
$\widetilde\theta_2$ are two ambient representatives of the same smooth
form on $X$, and put
\[
  \gamma:=\widetilde\theta_1-\widetilde\theta_2.
\]
On $X_{\reg}$ one has
\[
  \gamma_y|_{T_yX_{\reg}}=0.
\]
Let $v\in C_4(X,x)$. By Definition~\ref{def:fourth-whitney-cone},
there are points
\[
  x_\nu\in X_{\reg},\qquad x_\nu\to x,
\]
and vectors
\[
  v_\nu\in T_{x_\nu}X_{\reg},\qquad v_\nu\to v.
\]
Hence
\[
  \gamma_{x_\nu}(v_\nu,\overline{v_\nu})=0
\]
for every $\nu$. By the smoothness of $\gamma$,
\[
  \gamma_x(v,\overline v)
  =
  \lim_{\nu\to\infty}
  \gamma_{x_\nu}(v_\nu,\overline{v_\nu})
  =0.
\]
Thus all ambient representatives induce the same quadratic Hermitian
form on $C_4(X,x)$. Notice that the same conclusion need not hold on the
whole Zariski tangent space $T_xX$.

Assume first that $\theta$ is positive as a current on $X$. By the
regular-locus criterion of Lemma~\ref{lem:positivity-regular-locus}, its restriction
\[
  \theta|_{X_{\reg}}
\]
is pointwise semipositive.

Let $v\in C_4(X,x)$. Choose $x_\nu$ and $v_\nu$ as above. Since
$\widetilde\theta$ represents $\theta$, for every $\nu$ we have
\[
  \widetilde\theta_{x_\nu}
  (v_\nu,\overline{v_\nu})
  =
  \theta_{x_\nu}
  (v_\nu,\overline{v_\nu})
  \geq0.
\]
Passing to the limit gives
\[
  \widetilde\theta_x(v,\overline v)
  =
  \lim_{\nu\to\infty}
  \widetilde\theta_{x_\nu}
  (v_\nu,\overline{v_\nu})
  \geq0.
\]
Since both the local embedding and the ambient representative
$\widetilde\theta$ were arbitrary, $\theta$ is $C_4(X,x)$-positive.
Since $x$ was arbitrary, this holds at every point of $X$.

Conversely, suppose that $\theta$ is $C_4(X,x)$-positive at every
$x\in X$. Let $x\in X_{\reg}$. Since $X$ is a complex manifold near
$x$, the tangent spaces vary continuously there, and
\[
  C_4(X,x)=T_xX.
\]
Therefore, for every $v\in T_xX$,
\[
  \theta_x(v,\overline v)\geq0.
\]
It follows that
\[
  \theta|_{X_{\reg}}
\]
is pointwise semipositive. Applying again
Lemma~\ref{lem:positivity-regular-locus}, we conclude that $\theta$ is a
positive current on $X$.

\end{proof}

\begin{proof}[Proof of Theorem \ref{thm-tangentialpositive}]
Write $x:=x_0$. Assume, by contradiction, that $\theta$ is
positive as a smooth form.

By the definition of positivity for smooth forms, there exist a local
embedding
\[
  \jmath_1:(X,x)\hookrightarrow(\Omega_1,0)
\]
and a smooth semipositive real $(1,1)$-form $\eta$ on $\Omega_1$ whose
restriction to $X$ is $\theta$. In particular, the Hermitian form
\[
  H:=\eta_x|_{T_xX}
\]
is semipositive on $T_xX$.

By hypothesis, there is another local embedding
\[
  \jmath_2:(X,x)\hookrightarrow(\Omega_2,0)
\]
and a smooth extension $\theta'$ of $\theta$ such that
\[
  H':=\theta'_x|_{T_xX}
\]
is not semipositive.

We claim that
\[
  H(v,v)=H'(v,v)
  \qquad\text{for every }v\in C_4(X,x).
  \tag{*}
\]
To see this, we may replace the two embeddings by their product
embedding
\[
  \jmath=(\jmath_1,\jmath_2):
  (X,x)\hookrightarrow(\Omega_1\times\Omega_2,(0,0)).
\]
The forms
\[
  \pi_1^*\eta
  \qquad\text{and}\qquad
  \pi_2^*\theta'
\]
are then two smooth ambient extensions of the same smooth form $\theta$.

Let $v\in C_4(X,x)$. By
Definition~\ref{def:fourth-whitney-cone}, there exist
\[
  x_\nu\in X_{\reg},\qquad x_\nu\to x,
\]
and
\[
  v_\nu\in T_{x_\nu}X_{\reg},\qquad v_\nu\to v.
\]
Since the two ambient forms restrict to the same form $\theta$ on
$X_{\reg}$, we have
\[
  (\pi_1^*\eta)_{x_\nu}(v_\nu,\overline{v_\nu})
  =
  (\pi_2^*\theta')_{x_\nu}(v_\nu,\overline{v_\nu})
\]
for every $\nu$. Passing to the limit yields
\[
  H(v,\overline v)=H'(v,\overline v),
\]
which proves \textup{(*)}.

Now regard $T_xX$ as a complex vector space, and let $Z$ be an
irreducible component of $C_4(X,x)$ which is not contained in any complex
hyperplane. Its regular locus $Z_{\reg}$ is a connected complex
submanifold and is dense in $Z$; hence $Z_{\reg}$ is not contained in any
complex hyperplane either. Since \textup{(*)} holds on $Z_{\reg}$,
Lemma~\ref{lem:hermitian-uniqueness-cone}, applied to $Z_{\reg}$, $H$, and
$H'$, gives
\[
  H=H'
  \qquad\text{on }T_xX.
\]
But $H$ is semipositive, whereas by hypothesis $H'$ is not
semipositive. This is a contradiction.
Therefore $\theta$ cannot be positive as a smooth form.
\end{proof}

\begin{proof}[Proof of Theorem \ref{thm-tangentialpositive2}]
Set
\[
  V:=T_{x_0}X,
  \qquad
  C:=C_4(X,x_0),
  \qquad
  d:=\dim_{\C}V.
\]
By assumption~\textup{(i)},
\[
  \dim_{\C}C<d,
\]
and hence $C$ is a proper closed complex cone in $V$.

Choose a K\"ahler form $\omega$ on $X$. After shrinking around $x_0$,
choose a local embedding
\[
  \jmath:U\hookrightarrow\Omega\subset\C^N,
  \qquad \jmath(x_0)=0,
\]
and a smooth strictly plurisubharmonic function $\rho$ on $\Omega$ such
that
\begin{equation}\label{eq:thm2-local-potential-omega}
  \omega|_U=\ddc(\rho|_U).
\end{equation}
We identify $U$ with its image and, by
Lemma~\ref{lem:embedded-zariski-tangent}, identify $V$ with a complex
linear subspace of $\C^N$. We equip $V$ with the restriction of the
Euclidean norm.

The Hermitian separation lemma of Section~\ref{sec-cone-dauam}, applied
to $C\subset V$, gives a Hermitian form $H$ on $V$ and a constant $c>0$
such that
\begin{equation}\label{eq:thm2-H-positive-on-C4}
  H(v, v)\geq c\lVert v\rVert^2
  \qquad\text{for every }v\in C,
\end{equation}
and $H$ has signature $(d-1,1)$. Extend $H$ to a Hermitian form
$\widetilde H$ on $\C^N$, and let $q$ be a real Hermitian quadratic
polynomial on $\C^N$ such that the constant real $(1,1)$-form $\ddc q$
has associated Hermitian form $\widetilde H$.

We first pass from strict positivity on $C$ to uniform strict positivity
on the nearby tangent spaces of the regular locus. After shrinking $U$,
one has
\begin{equation}\label{eq:thm2-q-nearby-tangents}
  (\ddc q)_y(\xi,\overline\xi)
  \geq \frac{c}{2}\lVert\xi\rVert^2
\end{equation}
for every $y\in U\cap X_{\reg}$ and every
$\xi\in T_yX_{\reg}$. Indeed, if this failed in every neighborhood of
$x_0$, there would be sequences
\[
  y_\nu\in X_{\reg},\qquad y_\nu\longrightarrow x_0,
\]
and unit vectors $\xi_\nu\in T_{y_\nu}X_{\reg}$ such that
\[
  (\ddc q)_{y_\nu}(\xi_\nu,\overline{\xi_\nu})<\frac c2.
\]
After passing to a subsequence, $\xi_\nu\to\xi$ for some unit vector
$\xi\in\C^N$. By Definition~\ref{def:fourth-whitney-cone},
$\xi\in C\subset V$. Since $\widetilde H|_V=H$, passage to the limit
contradicts \eqref{eq:thm2-H-positive-on-C4}.
Let
\[
  \omega_{\mathrm{euc}}:=\ddc\lVert z\rVert^2.
\]
Since $\omega$ is represented by the smooth ambient form $\ddc\rho$, we
may shrink $U$ once more and choose $A_0>0$ such that
\[
  \omega\leq A_0\omega_{\mathrm{euc}}|_X
  \qquad\text{on }U\cap X_{\reg}.
\]
Consequently, \eqref{eq:thm2-q-nearby-tangents} gives
\begin{equation}\label{eq:thm2-q-dominates-omega}
  \ddc(q|_X)\geq\delta_0\omega
  \qquad\text{on }U\cap X_{\reg},
  \qquad
  \delta_0:=\frac{c}{2A_0}>0.
\end{equation}
We now prepare the two potentials for gluing. Since $\rho$ is strictly
plurisubharmonic on $\Omega$, we can choose $\lambda>0$ so small that
\[
  \ddc(\rho-\lambda q)_0
\]
is positive definite on $\C^N$. Adding to $\rho$ the real part of a
holomorphic polynomial of degree at most two does not change
\eqref{eq:thm2-local-potential-omega}. We may therefore arrange that
\begin{equation}\label{eq:thm2-normalized-difference}
  \rho(z)-\lambda q(z)
  =G(z,z)+O(\lVert z\rVert^3),
\end{equation}
where $G$ is a positive definite Hermitian form on $\C^N$. It follows
that, after shrinking $\Omega$, there exist $a,b>0$ and $R_0>0$ such
that
\begin{equation}\label{eq:thm2-quadratic-comparison}
  a\lVert z\rVert^2
  \leq \rho(z)-\lambda q(z)
  \leq b\lVert z\rVert^2
  \qquad\text{for }\lVert z\rVert\leq R_0.
\end{equation}
We choose $R_0$ so that $\overline B(0,R_0)\Subset\Omega$ and all the
preceding estimates hold over $B(0,R_0)$.

Choose radii $0<r<R<R_0$ satisfying
\[
  br^2<aR^2,
\]
and choose $\kappa\in\mathbb R$ such that
\begin{equation}\label{eq:thm2-choice-kappa}
  br^2<\kappa<aR^2.
\end{equation}
Then \eqref{eq:thm2-quadratic-comparison} gives
\begin{equation}\label{eq:thm2-inner-branch}
  \lambda q+\kappa>\rho
  \qquad\text{on }\overline B(0,r),
\end{equation}
and
\begin{equation}\label{eq:thm2-outer-branch}
  \lambda q+\kappa<\rho
  \qquad\text{on }\{\lVert z\rVert=R\}.
\end{equation}
The latter inequality remains strict on a collar of the sphere.
Choose a smooth even convex function
$\chi_\varepsilon:\mathbb R\to\mathbb R$ such that
\[
  \chi_\varepsilon(t)=|t|
  \quad\text{for }|t|\geq\varepsilon,
  \qquad
  |\chi_\varepsilon'|\leq1,
\]
and set
\[
  M_\varepsilon(s,t)
  :=\frac{s+t+\chi_\varepsilon(s-t)}{2}.
\]
Thus $M_\varepsilon$ is smooth, convex, nondecreasing in each variable,
translation equivariant, and equal to $\max\{s,t\}$ whenever
$|s-t|\geq\varepsilon$. By compactness in
\eqref{eq:thm2-inner-branch} and on a closed outer collar, we may choose
$\varepsilon>0$ smaller than both positive gaps. Define
\[
  \psi:=M_\varepsilon(\lambda q+\kappa,\rho)\big|_{X\cap B(0,R)}.
\]
Then
\begin{equation}\label{eq:thm2-psi-branches}
  \psi=\lambda q+\kappa
  \quad\text{near }X\cap\overline B(0,r),
  \qquad
  \psi=\rho
  \quad\text{near }X\cap\partial B(0,R).
\end{equation}
We claim that
\begin{equation}\label{eq:thm2-psi-lower-bound}
  \ddc\psi\geq\delta_1\omega
  \qquad\text{on }X_{\reg}\cap B(0,R),
  \qquad
  \delta_1:=\min\{\lambda\delta_0,1\}>0.
\end{equation}
Indeed, for arbitrary smooth functions $f,g$, direct differentiation
gives
\begin{align*}
  \ddc M_\varepsilon(f,g)
  ={}&a_\varepsilon(f-g)\,\ddc f
     +\bigl(1-a_\varepsilon(f-g)\bigr)\,\ddc g\\
   &+\frac12\chi_\varepsilon''(f-g)
     \,d(f-g)\wedge d^c(f-g),
\end{align*}
where
\[
  a_\varepsilon(t):=\frac{1+\chi_\varepsilon'(t)}2\in[0,1].
\]
The last term is semipositive. Applying the formula with
$f=\lambda q+\kappa$ and $g=\rho$, and using
\eqref{eq:thm2-local-potential-omega} and
\eqref{eq:thm2-q-dominates-omega}, proves
\eqref{eq:thm2-psi-lower-bound}.

We also verify that $\psi$ is a psh function on the singular space. By
\eqref{eq:thm2-q-dominates-omega} and the fact that $q$ is continuous, $q|_U$ is plurisubharmonic on
$U\cap X_{\reg}$. This combined with the continuity of $q$ and the Forn\ae ss--Narasimhan theorem
\citep[Theorem~5.3.1]{FornaessNarasimhan1980} shows that $q|_U$ is
plurisubharmonic on $U$. The function $\rho|_U$ is psh as well. The
holomorphic-disc characterization recalled in
Section~\ref{subsec:currents-on-spaces}, together with the convexity and
coordinatewise monotonicity of $M_\varepsilon$, now shows that $\psi$ is
psh on $X\cap B(0,R)$. By \eqref{eq:thm2-psi-branches}, the smooth function
\[
  \varphi:=\psi-\rho|_X
\]
on $X\cap B(0,R)$ vanishes on a neighborhood of the boundary. It
therefore extends by zero to a global smooth real-valued function on
$X$. Put
\begin{equation}\label{eq:thm2-global-form}
  \theta:=\omega+\ddc\varphi.
\end{equation}
Then $\theta$ is a global real closed smooth $(1,1)$-form. On
$X\cap B(0,R)$ one has
\[
  \theta=\ddc\psi,
\]
while outside this neighborhood one has $\theta=\omega$. Since $\psi$
is psh and agrees with the psh potential $\rho$ on a boundary collar,
these local potentials fit together. Thus $\theta$ has local potentials
in the sense of Section~\ref{subsec:currents-on-spaces}. By \eqref{eq:thm2-psi-lower-bound},
\[
  \theta\geq\delta_1\omega
  \qquad\text{on }X_{\reg}\cap B(0,R),
\]
and outside this neighborhood $\theta=\omega$. Hence, we get
\[
  \theta\geq\delta_1\omega
  \qquad\text{pointwise on }X_{\reg}.
\]
Lemma~\ref{lem:positivity-regular-locus} yields the same inequality as
currents on $X$. Therefore, $\theta$ is a K\"ahler current.
Finally, by \eqref{eq:thm2-psi-branches}, near $x_0$ one has
\[
  \theta=\lambda\ddc(q|_X).
\]
The local ambient extension $\lambda\ddc q$ induces on
$T_{x_0}X=V$ the Hermitian form $\lambda H$, which has signature
$(d-1,1)$ and is therefore not semipositive. Assumption~\textup{(ii)} and
Theorem~\ref{thm-tangentialpositive} imply that $\theta$ is not positive
as a smooth form. In particular, $\theta$ is not a K\"ahler form.
\end{proof}

\section{The compact analytic model}\label{sec:compact-model}

Let $G=\{1,\sigma\}\simeq\mathbf Z/2\mathbf Z$ act on $\Pj^3$ by
\begin{equation}\label{eq:projective-involution}
  \sigma[z_0:z_1:z_2:z_3]
  =[z_0:-z_1:-z_2:-z_3],
\end{equation}
and let
\[
  q:\Pj^3\longrightarrow X:=\Pj^3/G
\]
be the analytic quotient.  
Since $\Pj^3$ is compact, projective and normal, $X$ is a compact normal projective variety by
Subsection \ref{subsec:finite-quotient-normal}.  Moreover, since $q:\Pj^3\to X$ is finite and surjective and $\Pj^3$ is
irreducible, the variety $X$ is irreducible of dimension three;
in particular, it is of pure dimension. Since $X$ is projective,
it is a compact K\"ahler space.
Let
\[
  \widetilde U:=\{z_0\neq0\}\subset\Pj^3,
  \qquad
  p:=q([1:0:0:0])\in X,
  \qquad
  U:=q(\widetilde U).
\]
The quotient map is open: for every open set $V\subset\Pj^3$,
\[
  q^{-1}(q(V))=\bigcup_{g\in G}g(V)
\]
is open, and the quotient topology therefore makes $q(V)$ open.  Hence
$U$ is an open neighborhood of $p$ in $X$.
Using the affine coordinates
\[
  w_i:=\frac{z_i}{z_0},\qquad 1\leq i\leq3,
\]
we identify $\widetilde U$ with $\C^3$.  The involution becomes
\begin{equation}\label{eq:affine-involution}
  w=(w_1,w_2,w_3)\longmapsto-w=(-w_1,-w_2,-w_3).
\end{equation}
Hence $U$ is the analytic quotient $\C^3/\{\pm1\}$, and $p$ is the image
of $0$.

Set
\[
  E:=\operatorname{Sym}_3(\C)
  =\left\{
  \begin{pmatrix}
    a_{11}&a_{12}&a_{13}\\
    a_{12}&a_{22}&a_{23}\\
    a_{13}&a_{23}&a_{33}
  \end{pmatrix}:a_{ij}\in\C
  \right\}.
\]
Thus $E$ is a six-dimensional complex vector space with coordinates
\[
  a_{11},a_{12},a_{13},a_{22},a_{23},a_{33}.
\]
Consider the holomorphic map
\begin{equation}\label{eq:nu-map}
  \nu:\C^3\longrightarrow E,
  \qquad
  w\longmapsto ww^t=(w_iw_j)_{i,j},
\end{equation}
and the closed analytic subset
\begin{equation}\label{eq:def-Y}
  Y:=\{A\in E:\rank A\leq1\}.
\end{equation}
The set $Y$ is analytic because the rank condition is equivalent to the
vanishing of the $2\times2$ minors of $A$.

We equip $E$ with the Hilbert--Schmidt norm
\begin{equation}\label{eq:HS-norm}
  \lVert A\rVert_{\HS}^{2}:=\tr(AA^*).
\end{equation}

\begin{lemma}\label{lem:image-nu}
The image of $\nu$ is $Y$, and its fibers are
\[
  \nu^{-1}(ww^t)=\{w,-w\}\quad(w\neq0),
  \qquad
  \nu^{-1}(0)=\{0\}.
\]
\end{lemma}

\begin{proof}
Clearly $ww^t$ has rank at most one.  Conversely, let $A\neq0$ be a
symmetric matrix of rank one.  We can write $A=uv^t$ with $u,v\neq0$.
The symmetry of $A$ means
\[
  u_i v_j=u_j v_i\qquad\text{for all }i,j.
\]
Choosing an index $k$ with $u_k\neq0$ shows that
$v_j=(v_k/u_k)u_j$ for every $j$.  Thus $v=cu$ for some $c\neq0$.
Choose a square root $\sqrt c$ in $\C$ and put $w=\sqrt c\,u$; then
$A=ww^t$.  This proves $\nu(\C^3)=Y$. The assertion concerning the fibers of $\nu$ is trivial.
%If $uu^t=ww^t\neq0$, choose $k$ with $w_k\neq0$.  Then
%$u_k^2=w_k^2$, so $u_k=\varepsilon w_k$ for some
%$\varepsilon\in\{\pm1\}$.  From $u_ju_k=w_jw_k$ we obtain
%$u_j=\varepsilon w_j$ for all $j$.  Thus $u=\pm w$.  Finally,
%$ww^t=0$ forces $w_i^2=0$ for every $i$, hence $w=0$.
\end{proof}

The map $\nu$ is invariant under $w\mapsto-w$, so the defining property of
the analytic quotient gives a holomorphic map
\begin{equation}\label{eq:induced-nu}
  \bar\nu:\C^3/\{\pm1\}\longrightarrow Y
\end{equation}
%Thus by Osgood's theorem (\cite[Chapter 5]{Narasimhan-book}), the map $\bar \nu$ is biholomorphic. We also see from this that $\dim Y=3$.

The next proposition proves directly, using holomorphic power series, that
this is an isomorphism of complex spaces.

\begin{proposition}\label{prop:analytic-local-isomorphism}
The map $\bar\nu$ in \eqref{eq:induced-nu} is biholomorphic.  Consequently,
\begin{equation}\label{eq:local-isomorphism}
  U\simeq Y,
  \qquad
  (X,p)\simeq(Y,0)
\end{equation}
as complex spaces and as analytic germs, respectively.
\end{proposition}

We see from this that $\dim Y=3$.

\begin{proof}
Lemma~\ref{lem:image-nu} shows first that $\bar\nu$ is bijective.  We check
the complex structures, paying particular attention to the origin.

Let $f$ be a germ at $0\in\C^3$ which is invariant under $w\mapsto-w$.
Write its convergent Taylor series as
\[
  f(w)=\sum_{\alpha\in\mathbf N^3}c_\alpha w^\alpha.
\]
The identity $f(-w)=f(w)$ implies
\[
  c_\alpha=0\qquad\text{whenever }|\alpha|\text{ is odd}.
\]
Every monomial $w^\alpha$ of even total degree can be written as a product
of quadratic monomials $w_iw_j$: simply list the $|\alpha|$ factors and
pair them.  After choosing one such pairing for each $\alpha$, write
\[
  w^\alpha=\prod_{i\leq j}(w_iw_j)^{\beta_{ij}(\alpha)},
  \qquad
  \sum_{i\leq j}\beta_{ij}(\alpha)=\frac{|\alpha|}{2}.
\]
Consider the power series in six variables
\begin{equation}\label{eq:invariant-series}
  F(a):=
  \sum_{|\alpha|\ {\rm even}}
  c_\alpha\prod_{i\leq j}a_{ij}^{\beta_{ij}(\alpha)}.
\end{equation}
This series is genuinely convergent near $0$.  Indeed, choose $r>0$ such
that the Taylor series of $f$ converges absolutely for $|w_i|\leq r$.
If $|a_{ij}|<r^2$ for all $i,j$, then
\[
  \left|
  \prod_{i\leq j}a_{ij}^{\beta_{ij}(\alpha)}
  \right|
  \leq r^{|\alpha|},
\]
so absolute convergence of the Taylor series for $f$ implies absolute
convergence of \eqref{eq:invariant-series}.  By construction,
\[
  f(w)=F\bigl((w_iw_j)_{i,j}\bigr)=F(\nu(w)).
\]
We have therefore proved that every invariant holomorphic germ at $0$ is
the pullback by $\nu$ of a holomorphic germ on $E$.

Conversely, two holomorphic germs on $E$ have the same pullback by $\nu$
exactly when their difference vanishes on $\nu(\C^3)=Y$.  This statement
also holds at the level of germs: the identity
\[
  \lVert\nu(w)\rVert_{\HS}=\lVert w\rVert^2
\]
shows that $\nu$ maps a sufficiently small ball around $0$ onto a
neighborhood of $0$ in $Y$.  It follows from
the definition of the structure sheaf of an analytic subspace that
\begin{equation}\label{eq:germs-identification}
  \mathcal O_{Y,0}
  \simeq
  \mathcal O_{\C^3,0}^{\{\pm1\}}
  =\mathcal O_{\C^3/\{\pm1\},[0]}.
\end{equation}
Thus $\bar\nu$ is biholomorphic at the origin.

For completeness, let $A=ww^t\neq0$ and choose $k$ with $w_k\neq0$.
Then $a_{kk}\neq0$.  On the part of $Y$ where $a_{kk}\neq0$, every entry
is recovered from the three entries in the $k$-th column by
\begin{equation}\label{eq:rank-one-chart}
  a_{ij}=\frac{a_{ik}a_{jk}}{a_{kk}}.
\end{equation}
After choosing one of the two local holomorphic square roots of $a_{kk}$,
we recover
\[
  w_k=\sqrt{a_{kk}},
  \qquad
  w_j=\frac{a_{jk}}{w_k}.
\]
Changing the square root changes $w$ to $-w$, which is the same quotient
point.  Hence $\bar\nu$ also has a holomorphic inverse near every nonzero
point.  This proves that it is biholomorphic everywhere and establishes
\eqref{eq:local-isomorphism}.
\end{proof}

We determine now the  tangent space $T_0Y$ of $Y$ at 0. Recall that
\begin{equation}\label{eq:analytic-tangent-definition}
  T_0Y
  :=\{\xi\in E:dH_0(\xi)=0
       \text{ for every holomorphic germ $H$ vanishing on $Y$}\}.
\end{equation}
Let $H$ be such a germ.  For every $w\in\C^3$, the curve
\[
  \gamma_w(t):=tww^t
\]
lies in $Y$.  Hence $H(\gamma_w(t))=0$, and differentiation at $t=0$
gives
\begin{equation}\label{eq:dH-rank-one}
  dH_0(ww^t)=0.
\end{equation}
The rank-one symmetric matrices span all of $E$: the diagonal matrices are
$e_i e_i^t$, and for $i\neq j$,
\begin{equation}\label{eq:rank-one-span}
  e_i e_j^t+e_j e_i^t
  =(e_i+e_j)(e_i+e_j)^t-e_i e_i^t-e_j e_j^t.
\end{equation}
It follows from \eqref{eq:dH-rank-one} that $dH_0$ vanishes on all of $E$.
Since this is true for every holomorphic equation $H$, definition
\eqref{eq:analytic-tangent-definition} yields
\begin{equation}\label{eq:analytic-tangent}
  T_0Y=E,
  \qquad
  \dim_{\C}T_0Y=6.
\end{equation}
Geometrically, $Y$ is a cone and its tangent cone at the vertex is
\[
  C_0(Y,0)=Y,
\]
which has dimension three. Furthermore, we see that  $ C_0(Y,0) \backslash \{0\}$ is a connected submanifold and  its linear span is all of the
six-dimensional space $E$, by \eqref{eq:rank-one-span}.  This gap between
the nonlinear tangent cone and its linear span is the essential feature
used in our example.

\begin{lemma}\label{lem:regular-locus}
The regular locus of $Y$ is $Y_{\reg}=Y\setminus\{0\}$.  If $w\neq0$ and
$A=ww^t$, then
\begin{equation}\label{eq:tangent-space-Y}
  T_A Y
  =\{wv^t+vw^t:v\in\C^3\}.
\end{equation}
In particular, every $\xi\in T_A Y$ has rank at most two.
\end{lemma}

\begin{proof}
The action of $\{\pm1\}$ on $\C^3\setminus\{0\}$ is free.  For every
$w\neq0$, one can choose a sufficiently small neighborhood $V$ of $w$
such that $V\cap(-V)=\varnothing$.  The quotient map restricts to a
biholomorphism from $V$ onto its image.  Since $Y$ is biholomorphic to
$\C^3/\{\pm1\}$, it follows that every point
of $Y\setminus\{0\}$ is smooth.  On the other hand, $0$ is singular:
the local dimension of $Y$ is three, whereas
$\dim_{\C}T_0Y=6$ by \eqref{eq:analytic-tangent}.  This proves the assertion
about the regular locus.

Differentiating the parametrization $\nu(w)=ww^t$ gives
\begin{equation}\label{eq:differential-nu}
  d\nu_w(v)=wv^t+vw^t.
\end{equation}
For $w\neq0$ this differential is injective.  Indeed, choose an index $k$
with $w_k\neq0$.  If $wv^t+vw^t=0$, its $(k,k)$-entry gives
$2w_kv_k=0$, hence $v_k=0$; its $(j,k)$-entry then gives
$v_jw_k=0$ for every $j$, hence $v=0$.  Both $\C^3$ and $T_A Y$ have
dimension three, so \eqref{eq:differential-nu} identifies $\C^3$ with
$T_A Y$ and proves \eqref{eq:tangent-space-Y}.  Finally, each summand
$wv^t$ and $vw^t$ has rank at most one, and therefore
  $\rank(wv^t+vw^t)\leq2$.
\end{proof}

The preceding tangent-space formula also makes the three cones from
Subsection~\ref{subsec:tangent-cones} completely explicit in this example.

\begin{lemma}[The fourth Whitney cone of $Y$]
\label{lem:C4-of-Y}
At the vertex one has
\[
  C_0(Y,0)=Y=\{A\in E:\rank A\leq1\},
  \qquad
  C_4(Y,0)=\{A\in E:\rank A\leq2\}.
\]
Consequently,
\[
  \dim_{\C}C_0(Y,0)=3,
  \qquad
  \dim_{\C}C_4(Y,0)=5,
  \qquad
  \dim_{\C}T_0Y=6.
\]
\end{lemma}

\begin{proof}
The equality $C_0(Y,0)=Y$ follows because $Y$ is invariant under complex
dilations.  By Lemma~\ref{lem:regular-locus}, the tangent space at
$ww^t\neq0$ is
\[
  L_w:=\{wv^t+vw^t:v\in\C^3\}.
\]
The space $L_w$ depends only on $[w]\in\Pj^2$.  Since $\Pj^2$ is compact,
the union of these planes is closed, and Definition~\ref{def:fourth-whitney-cone}
gives
\[
  C_4(Y,0)=\bigcup_{[w]\in\Pj^2}L_w.
\]
Every matrix in $L_w$ has rank at most two, so the union is contained in
$\{A:\rank A\leq2\}$.

Conversely, let $A$ be a complex symmetric matrix of rank at most two. By Autonne–Takagi factorization, there exist a unitary matrix $U$ and  
$\sigma_1,\sigma_2\geq0$ such that
\[
  A=U\,\operatorname{diag}(\sigma_1,\sigma_2,0)\,U^t.
\]
Set
\[
  w=U(\sqrt{\sigma_1},\,\sqrt{-1}\sqrt{\sigma_2},\,0)^t,
  \qquad
  v=\frac12U(\sqrt{\sigma_1},\,-\sqrt{-1}\sqrt{\sigma_2},\,0)^t.
\]
A direct multiplication gives $A=wv^t+vw^t$, so $A\in L_w$; the zero
matrix is automatic.  This proves the asserted description of $C_4(Y,0)$.
The rank-at-most-two locus is the determinant hypersurface in the
six-dimensional vector space $E$, and hence has dimension five. 
\end{proof}

To summarize, we obtain the following result.

\begin{corollary}\label{cor:compact-projective-example}
Let $X$ be as above. 
Then \(X\) is a compact normal projective variety, and there exists a
real closed smooth \((1,1)\)-form \(\theta\) on \(X\), having local smooth
potentials, such that \(\theta\) is a K\"ahler current but is not
positive as a smooth form. In particular, \(\theta\) is not a
K\"ahler form.
\end{corollary}

\begin{proof} 
We already know that the quotient \(X\)
is a compact normal pure-dimensional projective variety, and  the germ
\((X,p)\) is biholomorphic to the germ \((Y,0)\), where
\[
  Y=\{A\in E:\rank A\leq1\},
  \qquad
  E=\operatorname{Sym}_3(\C).
\]
Since the tangent space and the fourth Whitney cone are preserved by
biholomorphisms of germs, this identification induces
\[
  T_pX\simeq T_0Y=E
\]
and
\[
  C_4(X,p)\simeq C_4(Y,0).
\]
Lemma~\ref{lem:C4-of-Y} gives
\[
  C_4(Y,0)
  =
  D_2
  :=
  \{A\in E:\rank A\leq2\}
\]
and
\[
  \dim_{\C}C_4(Y,0)=5
  <
  6=\dim_{\C}T_0Y.
\]
Consequently,
\[
  \dim_{\C}C_4(X,p)<\dim_{\C}T_pX.
\]

It remains to verify that \(D_2\) is irreducible and is not contained
in any complex hyperplane of \(E\). Consider the polynomial map
\[
  \Phi:
  \operatorname{Mat}_{3\times2}(\C)
  \longrightarrow E,
  \qquad
  B\longmapsto BB^t.
\]
Every matrix in the image of \(\Phi\) is symmetric and has rank at
most two. Conversely, let \(A\in D_2\). By the Autonne--Takagi
factorization, there exist a unitary matrix \(U\) and real numbers
\(\sigma_1,\sigma_2\geq0\) such that
\[
  A
  =
  U\operatorname{diag}(\sigma_1,\sigma_2,0)U^t.
\]
Setting
\[
  B
  :=
  U
  \begin{pmatrix}
    \sqrt{\sigma_1}&0\\
    0&\sqrt{\sigma_2}\\
    0&0
  \end{pmatrix},
\]
we obtain \(A=BB^t\). Hence
\[
  \Phi\bigl(\operatorname{Mat}_{3\times2}(\C)\bigr)=D_2.
\]
The affine space \(\operatorname{Mat}_{3\times2}(\C)\) is
irreducible, and the image of an irreducible space under holomorphic maps is
irreducible. Since \(D_2\) is closed, it follows that \(D_2\) is an
irreducible complex algebraic cone.

Moreover, \(D_2\) contains all rank-one symmetric matrices \(ww^t\).
By \eqref{eq:rank-one-span}, these matrices span the whole vector
space \(E\). Therefore \(D_2\) is not contained in any complex
hyperplane of \(E\).

Thus \(C_4(X,p)\) is irreducible, is not contained in any complex
hyperplane of \(T_pX\), and satisfies
\[
  \dim_{\C}C_4(X,p)<\dim_{\C}T_pX.
\]
All the assumptions of Theorem~\ref{thm-tangentialpositive2} are
therefore satisfied at \(p\). Applying that theorem gives a real
closed smooth \((1,1)\)-form \(\theta\) on \(X\), having local
potentials, which is a K\"ahler current but is not positive as a
smooth form. In particular, \(\theta\) is not a K\"ahler form.
\end{proof}

\bibliographystyle{plainnat}
\bibliography{Positivity-of-smooth-currents}

\Addresses

\end{document}